\documentclass[10pt]{article}
 \font\smallit=cmti10

\usepackage{amssymb,latexsym,amsmath,epsfig,amsthm} 

\makeatletter

\renewcommand{\@seccntformat}[1]{\csname the#1\endcsname. }

\makeatother

\newtheorem{theorem}{Theorem}[section]
 \newtheorem{lemma}[theorem]{Lemma}
 
 \newtheorem{proposition}[theorem]{Proposition}
 \newtheorem{corollary}[theorem]{Corollary}
 \newtheorem{definition}[theorem]{Definition}
 
\begin{document}
\begin{center}
{\bf On $p$-adic Minkowski's Theorems}

 \vskip 30pt

{\bf Yingpu Deng}\\
 {\smallit  Key Laboratory of Mathematics Mechanization, NCMIS, Academy of Mathematics and Systems Science, Chinese Academy of Sciences, Beijing 100190, People's Republic of China}\\
{and}\\
 {\smallit School of Mathematical Sciences, University of Chinese Academy of Sciences, Beijing 100049, People's Republic of China}\\

 \vskip 10pt

 {\tt dengyp@amss.ac.cn}\\

 \end{center}

 \vskip 30pt

 \centerline{\bf Abstract} Dual lattice is an important concept of Euclidean lattices. In this paper, we first give the right definition of the concept of the dual lattice of a $p$-adic lattice from the duality theory of locally compact abelian groups. The concrete constructions of ``basic characters'' of local fields given in Weil's famous book ``Basic Number Theory'' help us to do so. We then prove some important properties of the dual lattice of a $p$-adic lattice, which can be viewed as $p$-adic analogues of the famous Minkowski's first, second theorems for Euclidean lattices. We do this simultaneously for local fields $\mathbb{Q}_p$ (the field of $p$-adic numbers) and $\mathbb{F}_p((T))$ (the field of formal power-series of one indeterminate with coefficients in the finite field with $p$ elements).
\vspace{0.5cm}



\noindent

 \pagestyle{myheadings}

 \thispagestyle{empty}
 \baselineskip=12.875pt
 \vskip 20pt

\section{Introduction}
Lattices, as discrete additive subgroups of Euclidean spaces, are important mathematical objects.
The study of lattices in Euclidean spaces has a long history, they are the central objects in the geometry of numbers, see for instance \cite{cas,sie}. Lattices in Euclidean spaces have many applications in some practical areas such as cryptology, see \cite{mg,reg}. In 2018, two new computational problems, the longest vector problem (LVP) and the closest vector problem (CVP) are introduced in \cite{deng1}, also see \cite{deng2}. These computational problems find applications for constructing cryptographic primitives, such as public-key encryption schemes and signature schemes, see \cite{deng3}. Later, these computational problems also find applications for doing $p$-adic Gram-Schmidt orthogonalization process, see \cite{deng4}.

Completely contrary to lattices in Euclidean spaces, in \cite{deng4,zdw}, they showed that there is always an orthogonal basis with respect to any given norm in any $p$-adic lattice of arbitrary rank. This is a completely different phenomenon comparing with Euclidean lattices. In \cite{zdw}, they also showed that the sorted sequence of the norms of any norm-orthogonal basis of any $p$-adic lattice is unique. Then they defined the concept of successive maxima relative to the concept of successive minima of Euclidean lattices.

Dual lattice is an important concept of Euclidean lattices. They play important role in transference theorems and algorithmic reductions of Euclidean lattices. In this paper, we derive the right formulation of the concept of the dual lattice of a $p$-adic lattice from the duality theory of locally compact abelian groups. The concrete constructions of ``basic characters'' of local fields given in Weil's famous book ``Basic Number Theory'' help us to do so.

This paper is organized as follows. In Section \ref{preli}, we recall the necessary preliminaries, including Euclidean lattices and their duals, $p$-adic lattices, LVP and CVP in $p$-adic lattices, and the duality theory of locally compact abelian groups. In Section \ref{dual}, we give the right definition of the dual lattice of a $p$-adic lattice, then we prove some important properties of the dual lattice of a $p$-adic lattice, which can be viewed as $p$-adic analogues of the famous Minkowski's first, second theorems for Euclidean lattices. In Section \ref{computconst}, we show that the constants in Section \ref{dual} are computable by invoking CVP algorithms. We do this simultaneously for lattices over local fields $\mathbb{Q}_p$ (the field of $p$-adic numbers) and $\mathbb{F}_p((T))$ (the field of formal power-series of one indeterminate with coefficients in the finite field with $p$ elements).
\section{Preliminaries}\label{preli}
\subsection{Euclidean lattices and their dual lattices}\label{euclidean}
We first recall the concept of a dual lattice of a lattice in an Euclidean space. Let $m$ be a positive integer, $\mathbb{R}^m$ be the set of $m$-tuples of real numbers. Let $n$ be a positive integer with $n\leq m$. Given $n$ $\mathbb{R}$-linearly independent column vectors $b_1,b_2,\dots,b_n\in\mathbb{R}^m$, the lattice generated by them is defined as
$$\mathcal{L}(b_1,b_2,\dots,b_n)=\left\{\sum_{i=1}^nx_ib_i:x_i\in\mathbb{Z},1\leq i\leq n\right\}.$$
We call $b_1,b_2,\dots,b_n$ a basis of the lattice. Let $B$ be the $m\times n$ matrix whose columns are $b_1,b_2,\dots,b_n$, then the lattice generated by $B$ is
$$\mathcal{L}(B)=\mathcal{L}(b_1,b_2,\dots,b_n)=\{Bx:x\in\mathbb{Z}^n\}.$$

For a lattice $\Lambda$, we define its dual lattice
$$\Lambda^*=\{y\in \mbox{span}(\Lambda):y\cdot x\in\mathbb{Z},\forall x\in\Lambda\}$$
where span$(\Lambda)$ denotes the subspace spanned by all vectors in $\Lambda$ and $y\cdot x=\sum_{i=1}^my_ix_i$, for $y^t=(y_1,\ldots,y_m)$ and $x^t=(x_1,\ldots,x_m)$. Here $y^t$ denotes the transpose of column vector $y$. This usual definition of dual lattice of an Euclidean lattice can be found in \cite{reg,cas,sie}.

\subsection{$p$-adic lattices}

$p$-adic lattices have also been studied in mathematical literatures, see \cite{wei}. For a prime $p$, $\mathbb{Q}_p$ denotes the field of $p$-adic numbers. Formally,
$$\mathbb{Q}_p=\left\{\sum_{i=t}^{\infty}a_ip^i:t\in\mathbb{Z},a_i\in\{0,1,\ldots,p-1\}\right\}.$$
For $a\in\mathbb{Q}_p,a\neq0$, write $a=\sum_{i=t}^{\infty}a_ip^i$ as above and $a_t\neq0$, we define its $p$-adic absolute value as
$$|a|_p=p^{-t}\mbox{ and }|0|_p=0.$$
Define $\mathbb{Z}_p=\{a\in\mathbb{Q}_p:|a|_p\leq1\}$. Thus, $\mathbb{Q}_p$ becomes an ultrametric space, and $\mathbb{Z}_p$ is its maximal compact subring. One can see
$$\mathbb{Z}_p=\left\{\sum_{i=0}^{\infty}a_ip^i:a_i\in\{0,1,\ldots,p-1\}\right\}.$$

Also, for a prime $p$, $\mathbb{F}_p((T))$ denotes the field of formal power-series
$$\mathbb{F}_p((T))=\left\{\sum_{i=t}^{\infty}f_iT^i:t\in\mathbb{Z},f_i\in\mathbb{F}_p\right\}$$
where $\mathbb{F}_p$ denotes the finite field with $p$ elements and $T$ is an indeterminate. For $f\in\mathbb{F}_p((T)),f\neq0$, write $f=\sum_{i=t}^{\infty}f_iT^i$ with $f_t\neq0$, we define its $p$-adic absolute value as
$$|f|_p=p^{-t}\mbox{ and }|0|_p=0.$$
Define $\mathbb{F}_p[[T]]=\{f\in\mathbb{F}_p((T)):|f|_p\leq1\}$. Thus, $\mathbb{F}_p((T))$ becomes an ultrametric space, and $\mathbb{F}_p[[T]]$ is its maximal compact subring. One can see
$$\mathbb{F}_p[[T]]=\left\{\sum_{i=0}^{\infty}f_iT^i:f_i\in\mathbb{F}_p\right\}.$$
We know that $\mathbb{Q}_p$ is a commutative field of characteristic 0 and $\mathbb{F}_p((T))$ is a commutative field of characteristic $p$.

Let $k$ be $\mathbb{Q}_p$ or $\mathbb{F}_p((T))$ and $\mathcal{O}_k$ be its maximal compact subring defined as above. Let $m$ be a positive integer, $k^m$ be the set of $m$-tuples of elements in $k$. Let $n$ be a positive integer with $n\leq m$. Given $n$ $k$-linearly independent column vectors $b_1,b_2,\dots,b_n\in k^m$, the $p$-adic lattice generated by them is defined as
$$\mathcal{L}(b_1,b_2,\dots,b_n)=\left\{\sum_{i=1}^nx_ib_i:x_i\in\mathcal{O}_k,1\leq i\leq n\right\}.$$
We call $b_1,b_2,\dots,b_n$ a basis of the lattice. Let $B$ be the $m\times n$ matrix whose columns are $b_1,b_2,\dots,b_n$, then the lattice generated by $B$ is
$$\mathcal{L}(B)=\mathcal{L}(b_1,b_2,\dots,b_n)=\{Bx:x\in\mathcal{O}_k^n\}.$$
$n$ is called the rank of the $p$-adic lattice, if $n=m$, the $p$-adic lattice is called of full rank.
$p$-adic lattices are compact subsets of $k^m$. $p$-adic lattices obviously can be defined in any vector space over $k$.

\begin{definition}[$N$-orthogonal basis]
	Let $V$ be a vector space over $k$ of finite dimension $n>0$, and let $N$ be a norm on $V$. (The definition of a norm can be found in Section \ref{lvpcvp}.) We call $\alpha_1,\ldots,\alpha_n$  an $N$-orthogonal basis of $V$ over $k$ if  $V$ can be decomposed  into the direct sum of $n$ 1-dimensional subspaces $V_i$'s $(1\leq i\leq n)$, such that
	$$N\left(\sum_{i=1}^nv_i\right)=\max_{1\leq i\leq n}N(v_i), \forall v_i\in V_i, i=1,\ldots,n,$$
	where  $V_i$ is spanned by $\alpha_i$. Two subspaces $U,W$ of $V$ is said to be $N$-orthogonal if the sum $U+W$ is a direct sum and it holds that $N(u+w)=\max(N(u),N(w))$ for all $u\in U,w\in W$.
\end{definition}

Weil \cite[p.26, prop.3]{wei} proved the existence of an $N$-orthogonal basis of any vector space over $k$ of finite dimension.

\begin{definition}
Let $V$ be a vector space over $k$ of finite dimension $m>0$, and let $N$ be a norm on $V$. If $\alpha_1,\ldots,\alpha_n$ is an $N$-orthogonal basis of the vector space spanned by a lattice $\mathcal{L}=\sum_{i=1}^{n}\mathcal{O}_k\alpha_i$, then we call $\alpha_1,\ldots,\alpha_n$ an $N$-orthogonal basis of the lattice $\mathcal{L}$.
\end{definition}

Given a norm $N$ on $V$, for a $p$-adic lattice $\mathcal{L}$ of rank $n$ in $V$, does it has an $N$-orthogonal basis? For $n=1$, the answer is obviously affirmative. For $n=2$, Deng \cite{deng4} showed that the answer is affirmative. Zhang et al. \cite{zdw} showed that the answer is also affirmative for any $p$-adic lattice with arbitrary rank. Notice that, in \cite{deng4,zdw}, they proved the results only for $k=\mathbb{Q}_p$, but the reasoning is obviously true for $k=\mathbb{F}_p((T))$.

\begin{theorem}\cite{deng4,zdw}\label{latticeorthobasis}
Let $V$ be a vector space over $k$, and let $N$ be a norm on $V$. Let $\mathcal{L}$ be a $p$-adic lattice in $V$ of rank $n>0$. Then $\mathcal{L}$ must have an $N$-orthogonal basis.
\end{theorem}

In \cite{zdw}, they also proved the following result.

\begin{theorem}\cite{zdw}
Let $V$ be a vector space over $k$, and let $N$ be a norm on $V$. Let $(\alpha_1,\ldots,\alpha_n)$ and $(\beta_1,\ldots,\beta_n)$ be two $N$-orthogonal bases of a $p$-adic lattice $\mathcal{L}$ of rank $n$ in $V$. Assume $N(\alpha_1)\geq N(\alpha_2)\geq\cdots\geq N(\alpha_n)$ and $N(\beta_1)\geq N(\beta_2)\geq\cdots\geq N(\beta_n)$. Then we have $N(\alpha_i)=N(\beta_i)$ for $1\leq i\leq n$.
\end{theorem}

\begin{proof}
Below, we give a shorter proof than that in \cite{zdw}. Set $\pi=p$ if $k=\mathbb{Q}_p$ or $\pi=T$ if $k=\mathbb{F}_p((T))$. Then $\pi$ is a uniformizer of $k$ with $|\pi|_p=p^{-1}$ and we have $\mathcal{O}_k/\pi\mathcal{O}_k=\mathbb{F}_p$.

For any vector $v\in\mathcal{L}$, write $v=\sum_{i=1}^na_i\alpha_i$ with $a_i\in\mathcal{O}_k$. Then we have
$$N(v)\leq\max_{1\leq i\leq n} N(a_i\alpha_i)\leq\max_{1\leq i\leq n} N(\alpha_i)=N(\alpha_1).$$
That is, $N(\alpha_1)$ is the biggest norm of all the vectors in the lattice $\mathcal{L}$. Similarly, $N(\beta_1)$ is also the biggest norm of all the vectors in the lattice $\mathcal{L}$. Thus $N(\alpha_1)=N(\beta_1)$.

Now we prove $N(\alpha_n)=N(\beta_n)$. Suppose not, without loss of generality, we may suppose $N(\alpha_n)>N(\beta_n)$. Write $\beta_n=\sum_{i=1}^na_i\alpha_i$ with $a_i\in\mathcal{O}_k$. If there is a $j\in\{1,\ldots,n\}$ with $a_j\in\mathcal{O}_k^{\times}=\{a\in\mathcal{O}_k:|a|_p=1\}$, since $(\alpha_1,\ldots,\alpha_n)$ is an $N$-orthogonal bases of $\mathcal{L}$, we have
$$N(\beta_n)=\max_{1\leq i\leq n}N(a_i\alpha_i)\geq N(a_j\alpha_j)=N(\alpha_j)>N(\beta_n).$$
This yields a contradiction. Thus $a_i\in\pi\mathcal{O}_k$ for all $1\leq i\leq n$. But then
$$\frac{\beta_n}{\pi}=\sum_{i=1}^n\frac{a_i}{\pi}\alpha_i\in\mathcal{L}.$$
Since $\beta_n$ is a basis vector of the lattice $\mathcal{L}$, this is impossible. Therefore, we must have $N(\alpha_n)=N(\beta_n)$.

We have proved that $N(\alpha_1)=N(\beta_1)$ and $N(\alpha_n)=N(\beta_n)$. Thus, for $n=1,2$, the statement in the theorem is true. Now we assume $n>2$.

Below we will prove $N(\alpha_j)=N(\beta_j)$ for all $1<j<n$. Suppose not, without loss of generality, we may suppose $N(\alpha_j)>N(\beta_j)$ for some $1<j<n$.

For each $j\leq i\leq n$, write
$$\beta_i=\sum_{l=1}^na_{il}\alpha_l\mbox{ with }a_{il}\in\mathcal{O}_k.$$
Since $N(\beta_i)<N(\alpha_j)\leq N(\alpha_{j-1})\leq\cdots\leq N(\alpha_1)$, similar to the above reasoning, we have
$$a_{il}\in\pi\mathcal{O}_k\mbox{ for all }1\leq l\leq j.$$
For an element $a\in\mathcal{O}_k$, we denote by $\overline{a}$ the image of $a$ in $\mathcal{O}_k/\pi\mathcal{O}_k=\mathbb{F}_p$. Thus, we have $n-j+1$ many $(n-j)$-tuples $(\overline{a_{i,j+1}},\ldots,\overline{a_{i,n}})\in\mathbb{F}_p^{n-j}$ for $j\leq i\leq n$. These $n-j+1$ many $(n-j)$-tuples in $\mathbb{F}_p^{n-j}$ must be linearly dependent over $\mathbb{F}_p$. That is, there are $b_j,b_{j+1},\ldots,b_n\in\mathcal{O}_k$, with $(\overline{b_j},\overline{b_{j+1}},\ldots,\overline{b_n})\neq0$ such that
$$\sum_{i=j}^n\overline{b_i}(\overline{a_{i,j+1}},\ldots,\overline{a_{i,n}})=0.$$
That is
$$\sum_{i=j}^nb_i(a_{i,j+1},\ldots,a_{i,n})\in\left(\pi\mathcal{O}_k\right)^{n-j}.$$
Hence
$$\sum_{i=j}^nb_i\beta_i\in\sum_{l=1}^n\pi\mathcal{O}_k\alpha_l=\pi\mathcal{L}(\alpha_1,\ldots,\alpha_n).$$
Thus
$$\sum_{i=j}^n\frac{b_i}{\pi}\beta_i\in\mathcal{L}.$$
Since at least one of the $b_j,b_{j+1},\ldots,b_n$ is in $\mathcal{O}_k^{\times}$ and $\beta_j,\beta_{j+1},\ldots,\beta_n$ are basis vectors of the lattice $\mathcal{L}$, this is impossible. Therefore, we must have $N(\alpha_j)=N(\beta_j)$ for all $1<j<n$. The proof of the theorem is complete.
\end{proof}

Due to this theorem, we can give the following definition.

\begin{definition}[successive maxima]
Let $V$ be a vector space over $k$, and let $N$ be a norm on $V$. Let $(\alpha_1,\ldots,\alpha_n)$ be an $N$-orthogonal bases of a $p$-adic lattice $\mathcal{L}$ of rank $n$ in $V$. Assume $N(\alpha_1)\geq N(\alpha_2)\geq\cdots\geq N(\alpha_n)$. We call $N(\alpha_i)(1\leq i\leq n)$ the $i$-th successive maxima of the lattice $\mathcal{L}$ with respect to the norm $N$. We denote it by $\overline{\lambda_i}(\mathcal{L})$.
\end{definition}

\subsection{LVP and CVP in $p$-adic Lattices}\label{lvpcvp}
The shortest vector problem (SVP) and the closest vector problem (CVP) are two famous computational problems in Euclidean lattices. In $p$-adic lattices, we can define the longest vector problem (LVP) and the closest vector problem (CVP) as follows. Let $k$ be $\mathbb{Q}_p$ or $\mathbb{F}_p((T))$. Let $V$ be a vector space over $k$. A norm $N$ on $V$ is a function
$$N:V\longrightarrow\mathbb{R}$$
such that:
\begin{enumerate}
\item[(i)] $N(v)\geq0,\forall v\in V, \mbox{ and }N(v)=0\mbox{ if and only if }v=0;$
\item[(ii)] $N(xv)=|x|_p\cdot N(v),\forall x\in k, v\in V;$
\item[(iii)] $N(v+w)\leq\max(N(v),N(w)), \forall v,w\in V.$
\end{enumerate}
Here, $|x|_p$ is the $p$-adic absolute value for $x\in k$.

The following proposition can be found in \cite[p.72, prop.]{rob}.

\begin{proposition}
Let $\Omega\subset V$ be a compact subset. (a) For every $a\in V-\Omega$, the set of norms $\{N(x-a):x\in\Omega\}$ is finite. (b) For every $a\in\Omega$, the set of norms $\{N(x-a):x\in\Omega-\{a\}\}$ is discrete in $\mathbb{R}_{>0}$.
\end{proposition}

Due to this proposition, we can define the longest vector problem (LVP) and the closest vector problem (CVP) in $p$-adic lattices in this general setting as in \cite{deng2}, and all results in \cite{deng2} obviously hold also for any norm $N$ if we can compute efficiently the norm $N(v)$ of any vector $v\in V$. We recall here the definition of CVP.

\begin{definition}
Let $V$ be a vector space over $k$. Let $N$ be a norm on $V$. Given a lattice $\mathcal{L}=\mathcal{L}(\alpha_1,\ldots,\alpha_n)$ and a target vector $t\in V$. The closest vector problem (CVP) is to find a lattice vector $v_0\in\mathcal{L}$ such that
$$N(t-v_0)=\min_{v\in\mathcal{L}}N(t-v):=\mbox{dist}(t,\mathcal{L}).$$
\end{definition}

\subsection{Duality theory for locally compact abelian groups}
To give the right definition of a dual lattice of a $p$-adic lattice, we recall the duality theory for locally compact abelian groups, see Chapter II, \S 5 of \cite{wei} and a complete treatment can be found in \cite{hr}.

Let $G$ be a locally compact abelian group. Set $U=\{z\in\mathbb{C}:|z|=1\}$. A character of $G$ is a continuous representation of $G$ into $U$. If $g^*$ is such a character, we write $\langle g,g^*\rangle_G$ for the value of $g^*$ at a point $g\in G$. Let $G^*$ be the set of all characters of $G$. $G^*$ has an abelian group structure and one can define a topology on $G^*$ such that $G^*$ is a locally compact group. We call $G^*$ the topological dual of $G$. If $H$ is any closed subgroup of $G$, we define $H_*=\{g^*\in G^*:\langle h,g^*\rangle_G=1,\forall h\in H\}$. $H_*$ is said to be associated with $H$ by duality, and it is isomorphic to the dual of $G/H$.

Let $k$ be $\mathbb{R}, \mathbb{Q}_p$ or $\mathbb{F}_p((T))$. Let $V$ be a vector space of finite dimension over $k$. We know that the topological dual $V^*$ of $V$ has also a vector-space structure over $k$. If $L$ is any closed subgroup of $V$, the subgroup $L_*$ of $V^*$ associated with $L$ by duality consists of the elements $v^*$ of $V^*$ such that $\langle v,v^*\rangle_V=1$ for all $v\in L$.

On the other hand, if $V$ is as above, we may consider its algebraic dual $V'$, which is the space of all $k$-linear forms on $V$. We denote by $[v,v']_V$ the value of the linear form $v'$ on $V$ at the point $v$ of $V$. The main theorem of duality theory is the following:

\begin{theorem}\cite[p.40, th.3]{wei}\label{duality}
Let $k$ be $\mathbb{R}, \mathbb{Q}_p$ or $\mathbb{F}_p((T))$. Let $V$ be a vector space of finite dimension over $k$ and let $\chi$ be a non-trivial character of the additive group of $k$. Then the formula
$$\langle v,v^*\rangle_V=\chi([v,v']_V) \mbox{ for all }v\in V$$
defines an isomorphism $v'\longmapsto v^*$ from the algebraic dual $V'$ onto the topological dual $V^*$ as vector-spaces over $k$.
\end{theorem}

\section{Dual lattices}\label{dual}

\subsection{Definition of dual lattices of $p$-adic lattices}
Keeping the notations in Theorem \ref{duality}, let $V$ be a vector space of dimension $m>0$ over $k$. We identify $V'$ with $V^*$ by the isomorphism in Theorem \ref{duality}, and we also identify $v'$ with the corresponding $v^*$ by that isomorphism. Let $e_1,\ldots,e_m$ be a basis of $V$ over $k$. Let $e_i'\in V'$ be defined by $e_i'(e_j)=\delta_{ij}(1\leq i,j\leq m)$, where $\delta_{ij}=1$ if $i=j$ and 0 otherwise. By the knowledge of linear algebra, we know that $e_1',\ldots,e_m'$ is a basis of $V'$ over $k$. For any $v'\in V'$, $v=\sum_{i=1}^mv_ie_i\in V$, we have $[v,v']_V=v'(v)=\sum_{i=1}^mv_iv'(e_i)$. Also, we have $v'=\sum_{i=1}^mv'(e_i)e_i'$. The correspondence
$$v'\longmapsto \sum_{i=1}^mv'(e_i)e_i$$
is a $k$-isomorphism from $V'$ onto $V$. Under this isomorphism, $e_j'$ corresponds to $e_j$ for $1\leq j\leq m$. We identify $V'$ with $V$ by this isomorphism. Obviously, the correspondence
$$v=\sum_{i=1}^mv_ie_i\longmapsto(v_1,\ldots,v_m)^t$$
is a $k$-isomorphism from $V$ onto $k^m$. We also identify $V$ with $k^m$ by this isomorphism. Thus
$v'\in V'$ is identified with $(v'(e_1),\ldots,v'(e_m))^t\in k^m$.

Therefore, without loss of generality, we may simply regard $V=V'=V^*=k^m$. Write $v^t=(x_1,\ldots,x_m),v'^t=(y_1,\ldots,y_m)$, then we have
$$[v,v']_V=\sum_{i=1}^mx_iy_i.$$
Now suppose $L$ is any closed subgroup of $V$. Letting $v\in L,v^*\in L_*$, since $\langle v,v^*\rangle_V=1$, by the formula in Theorem \ref{duality},
we have $[v,v']_V\in\mbox{Ker}(\chi)$, i.e.,
$$\sum_{i=1}^mx_iy_i\in\mbox{Ker}(\chi),\forall v\in L.$$

For $k=\mathbb{R}$, we put $\chi(t)=e^{2\pi it}$ for $t\in\mathbb{R}$, here $i=\sqrt{-1}$, hence Ker$(\chi)=\mathbb{Z}$. This explains the definition of dual lattices of Euclidean lattices in Section \ref{euclidean}.

For $k=\mathbb{Q}_p$ or $k=\mathbb{F}_p((T))$, the construction of ``basic character'' $\chi$ can be found in \cite[pp.66-68]{wei}.

When $k=\mathbb{Q}_p$, call $\mathbb{Q}^{(p)}$ the set of the elements $\xi$ of $\mathbb{Q}$ such that $|\xi|_{p'}\leq1$ for all the primes $p'$ other than $p$. Clearly this is a subring of $\mathbb{Q}$, consisting of the numbers of the form $p^{-n}a$ with $n\in\mathbb{Z}_{\geq0}$ and $a\in\mathbb{Z}$. By \cite[p.64, Lemma 1]{wei}, we have $\mathbb{Q}_p=\mathbb{Q}^{(p)}+\mathbb{Z}_p$ and $\mathbb{Q}^{(p)}\cap\mathbb{Z}_p=\mathbb{Z}$. For $x\in\mathbb{Q}_p$, write $x=\xi+a$ with $\xi\in\mathbb{Q}^{(p)}$ and $a\in\mathbb{Z}_p$. Put $\chi(x)=e^{2\pi i\xi}$. We have Ker$(\chi)=\mathbb{Z}_p$.

When $k=\mathbb{F}_p((T))$, for $f=\sum f_iT^i\in k$, put $\chi(f)=\zeta^{f_{-1}}$, where $\zeta$ is a complex primitive $p$th root of unity. Hence $f\in\mbox{Ker}(\chi)$ if and only if $f_{-1}=0$.

Now we can give the definition of dual lattice of a $p$-adic lattice.

\begin{definition}
Let $\Lambda$ be a $p$-adic lattice in $\mathbb{Q}_p^m$. We define its dual lattice
$$\Lambda^*=\{y\in \mbox{span}(\Lambda):y\cdot x\in\mathbb{Z}_p,\forall x\in\Lambda\}$$
where span$(\Lambda)$ denotes the subspace spanned by all vectors in $\Lambda$ and $y\cdot x=\sum_{i=1}^my_ix_i$, for $y^t=(y_1,\ldots,y_m)$ and $x^t=(x_1,\ldots,x_m)$.
\end{definition}

\begin{definition}
Let $\Lambda$ be a $p$-adic lattice in $\mathbb{F}_p((T))^m$. We define its dual lattice
$$\Lambda^*=\{y\in \mbox{span}(\Lambda):y\cdot x\in\mathbb{F}_p[[T]],\forall x\in\Lambda\}$$
where span$(\Lambda)$ denotes the subspace spanned by all vectors in $\Lambda$ and $y\cdot x=\sum_{i=1}^my_ix_i$, for $y^t=(y_1,\ldots,y_m)$ and $x^t=(x_1,\ldots,x_m)$.
\end{definition}

For the second definition, we provide an illustration. For $y\in\mbox{span}(\Lambda)$, we have $y\in\Lambda^*$ if and only if $y\cdot x\in \mbox{Ker}(\chi),\forall x\in\Lambda$. Write
$$y\cdot x=\sum_i f_iT^i.$$
For $j<0$, since $x\in\Lambda$ implies $T^{-j-1}x\in\Lambda$, we have $y\cdot T^{-j-1}x\in\mbox{Ker}(\chi)$. But
$$y\cdot T^{-j-1}x=\sum_i f_iT^{i-j-1},$$
the coefficient of $T^{-1}$ of the above expression is $f_j$. So we have $f_j=0$ for all $j<0$. Hence $y\cdot x\in\mathbb{F}_p[[T]]$.

Thus we have a unified definition of dual lattices of $p$-adic lattices.

\begin{definition}\label{padicdual}
Let $k=\mathbb{Q}_p$ or $k=\mathbb{F}_p((T))$. Let $\Lambda$ be a $p$-adic lattice in $k^m$. We define its dual lattice
$$\Lambda^*=\{y\in \mbox{span}(\Lambda):y\cdot x\in\mathcal{O}_k,\forall x\in\Lambda\}$$
where span$(\Lambda)$ denotes the subspace spanned by all vectors in $\Lambda$ and $y\cdot x=\sum_{i=1}^my_ix_i$, for $y^t=(y_1,\ldots,y_m)$ and $x^t=(x_1,\ldots,x_m)$.
\end{definition}

\subsection{Properties of $p$-adic dual lattices}
Let $k=\mathbb{Q}_p$ or $k=\mathbb{F}_p((T))$.

\begin{definition}
For a basis $B=(b_1,\ldots,b_n)\in k^{m\times n}$, define the dual basis $D=(d_1,\ldots,d_n)\in k^{m\times n}$ as the unique basis that satisfies $span(D)= span(B)$ and $B^tD=I_n$, where $I_n$ is the identity matrix of order $n$.
\end{definition}

Notice that $B^tD=I_n$ means $b_i\cdot d_j=\delta_{ij}$ where $\delta_{ij}=1$ if $i=j$ and 0 otherwise. So the existence and the uniqueness of dual basis is clear.

\begin{proposition}\label{latticedualbasis}
If $D$ is the dual basis of $B$ then $\mathcal{L}(B)^*=\mathcal{L}(D)$.
\end{proposition}

\begin{proof}
We first show that $\mathcal{L}(D)$ is contained in $\mathcal{L}(B)^*$. Any $x\in\mathcal{L}(B)$ can be written as $x=\sum_{i=1}^n a_ib_i$ with $a_i\in\mathcal{O}_k$. Therefore, for any $j$ we have
$$x\cdot d_j=\sum_{i=1}^n a_i(b_i\cdot d_j)=\sum_{i=1}^n a_i\delta_{ij}=a_j\in\mathcal{O}_k.$$
Thus we get $D\subseteq\mathcal{L}(B)^*$. From Definition \ref{padicdual}, it is clear that $\mathcal{L}(B)^*$ is an $\mathcal{O}_k$-module in $k^m$, so we have $\mathcal{L}(D)\subseteq\mathcal{L}(B)^*$. Conversely, take any $y\in\mathcal{L}(B)^*$. Since $y\in\mbox{span}(B)=\mbox{span}(D)$, we can write $y=\sum_{i=1}^na_id_i$ with $a_i\in k$. Since $y\cdot b_j=a_j$ for all $1\leq j\leq n$, $y\in\mathcal{L}(B)^*$ implies $a_j\in\mathcal{O}_k$. Thus $y\in\mathcal{L}(D)$. Hence $\mathcal{L}(B)^*\subseteq\mathcal{L}(D)$. This completes the proof.
\end{proof}

\begin{proposition}
For any $p$-adic lattice $\Lambda$, $(\Lambda^*)^*=\Lambda$.
\end{proposition}

\begin{proof}
Suppose $B$ is a basis of $\Lambda$, i.e., $\Lambda=\mathcal{L}(B)$. Let $D$ be the dual basis of $B$, i.e., $\mbox{span}(B)=\mbox{span}(D)$ and $B^tD=I_n$. By Proposition \ref{latticedualbasis}, we have $\Lambda^*=\mathcal{L}(D)$. Since $D^tB=I_n$, $B$ is the dual basis of $D$, so by Proposition \ref{latticedualbasis}, we have $(\Lambda^*)^*=\mathcal{L}(B)$. This completes the proof.
\end{proof}

Take a lattice $\Lambda=\mathcal{L}(B)$ of rank $n$ in $k^m$, $B\in k^{m\times n}$ is its basis. We know $B^tB\in\mbox{GL}_n(k)$. Let $A\in k^{m\times n} $ be its another basis. Thus $A=BC$ for some $C\in\mbox{GL}_n(\mathcal{O}_k)$. Hence $A^tA=C^tB^tBC$. Therefore,
$$|\mbox{det}(A^tA)|_p=|\mbox{det}(B^tB)|_p,$$
i.e., $|\mbox{det}(B^tB)|_p$ is an invariant of the lattice $\Lambda$. We have the following definition.

\begin{definition}
Suppose a lattice $\Lambda=\mathcal{L}(B)$ is of rank $n$ in $k^m$, $B\in k^{m\times n}$ is its basis. Define its determinant is
$$\mbox{det}(\Lambda)=|\mbox{det}(B^tB)|_p^{1/2}.$$
\end{definition}

Thus if $\Lambda=\mathcal{L}(B)$ is of full rank, then $\mbox{det}(\Lambda)=|\mbox{det}(B)|_p$.

\begin{proposition}\label{dualdet}
For any $p$-adic lattice $\Lambda$, $\mbox{det}(\Lambda^*)=1/\mbox{det}(\Lambda)$.
\end{proposition}

\begin{proof}
Suppose $\Lambda=\mathcal{L}(B)$ and $D$ is the dual basis of $B$. By Proposition \ref{latticedualbasis}, we have $\Lambda^*=\mathcal{L}(D)$. Since $B^tB(B^tB)^{-1}=I$, we have $D=B(B^tB)^{-1}$. Thus
$$\mbox{det}(\Lambda^*)=\sqrt{|\mbox{det}(D^tD)|_p}=\sqrt{|\mbox{det}(((B^tB)^{-1})^tB^tB(B^tB)^{-1})|_p}$$
$$=\sqrt{|\mbox{det}(((B^tB)^{-1})^t)|_p}=\frac{1}{\sqrt{|\mbox{det}(B^tB)|_p}}=\frac{1}{\mbox{det}(\Lambda)}.$$
\end{proof}

\begin{definition}
For $v=(v_1,\ldots,v_m)^t\in k^m$, define $M(v)=\max_{1\leq i\leq m}|v_i|_p$. Obviously, $M$ is a norm on $k^m$.
\end{definition}

We need the following lemma.

\begin{lemma}\cite[p.91, th.]{rob}\label{normequiv}
Let $N$ be a norm on $k^m$. Then the two norms $M$ and $N$ are equivalent, i.e., there are two positive constants $c_1,c_2$ such that
$$c_1N(v)\leq M(v)\leq c_2N(v),\forall v\in k^m.$$
\end{lemma}

Notice that, in \cite{rob}, the above lemma is proved only for $k=\mathbb{Q}_p$, but it is easy to see that the lemma also holds for $k=\mathbb{F}_p((T))$.

\begin{theorem}\label{padicminkowski}
(1) Suppose a lattice $\Lambda=\mathcal{L}(B)$ is of rank $n$ in $k^m$, $B\in k^{m\times n}$ is its basis. Then we have
$$\lambda_1\geq\left(\mbox{det}(\Lambda)\right)^{\frac{1}{n}},$$
where $\lambda_1$ is the biggest norm of all the vectors in $\Lambda$ with respect to the norm $M$.

(2) Suppose a lattice $\Lambda=\mathcal{L}(B)$ is of rank $n$ in $k^m$, $B\in k^{m\times n}$ is its basis. Let $N$ be a norm on $k^m$. Then we have
$$\lambda_1\geq c\cdot\left(\mbox{det}(\Lambda)\right)^{\frac{1}{n}},$$
where $\lambda_1$ is the biggest norm of all the vectors in $\Lambda$ with respect to the norm $N$ and $c$ is a positive constant depending on $N$.
\end{theorem}

\begin{proof}
Suppose $B=(b_1,\ldots,b_n)\in k^{m\times n}$. Then $B^tB$ is an $n\times n$-matrix, whose $(i,j)$-element is $b_i\cdot b_j$. Obviously, we have
$$|b_i\cdot b_j|_p\leq M(b_i)\cdot M(b_j)\leq\lambda_1^2.$$
Thus
$$\mbox{det}(\Lambda)=|\mbox{det}(B^tB)|_p^{1/2}\leq\lambda_1^n.$$
This proves (1). Below, we prove (2).

By Lemma \ref{normequiv}, two norms $M$ and $N$ are equivalent, i.e., there are two positive constants $c_1,c_2$ such that
$$c_1N(v)\leq M(v)\leq c_2N(v),\forall v\in k^m.$$
Suppose $B=(b_1,\ldots,b_n)\in k^{m\times n}$. Then $B^tB$ is an $n\times n$-matrix, whose $(i,j)$-element is $b_i\cdot b_j$. Obviously, we have
$$|b_i\cdot b_j|_p\leq M(b_i)\cdot M(b_j)\leq c_2^2N(b_i)N(b_j)\leq c_2^2\lambda_1^2.$$
Thus
$$\mbox{det}(\Lambda)=|\mbox{det}(B^tB)|_p^{1/2}\leq c_2^n\cdot\lambda_1^n.$$
One can put $c=1/c_2$.
\end{proof}

\textbf{Remark.} Theorem \ref{padicminkowski} can be viewed as a $p$-adic analogue of Minkowski's first theorem for Euclidean lattices, see \cite{reg,sie,cas}.

\begin{theorem}\label{1padictransfer}
(1) Suppose $\Lambda$ is any $p$-adic lattice of rank $n$ in $k^m$, $\Lambda^*$ is its dual. Then we have
$$\lambda_1(\Lambda)\cdot\lambda_1(\Lambda^*)\geq1,$$
where $\lambda_1(\Lambda)$, $\lambda_1(\Lambda^*)$ is the biggest norm of all the vectors in $\Lambda$, $\Lambda^*$ with respect to the norm $M$, respectively.

(2) Suppose $\Lambda$ is any $p$-adic lattice of rank $n$ in $k^m$, $\Lambda^*$ is its dual. Let $N$ be a norm on $k^m$. Then we have
$$\lambda_1(\Lambda)\cdot\lambda_1(\Lambda^*)\geq c^2,$$
where $\lambda_1(\Lambda)$, $\lambda_1(\Lambda^*)$ is the biggest norm of all the vectors in $\Lambda$, $\Lambda^*$ with respect to the norm $N$, respectively, and the positive constant $c$ is the same as in Theorem \ref{padicminkowski} (2).
\end{theorem}

\begin{proof}
This follows from Theorem \ref{padicminkowski} and Proposition \ref{dualdet}.
\end{proof}

\begin{theorem}\label{padictransfer}
(1) Suppose $\Lambda$ is any $p$-adic lattice of rank $n$ in $k^m$, $\Lambda^*$ is its dual. Then we have
$$\overline{\lambda_1}(\Lambda)\cdot\overline{\lambda_n}(\Lambda^*)\geq1,$$
where $\overline{\lambda_1}(\Lambda)=\lambda_1(\Lambda)$ is the first successive maxima of the lattice $\Lambda$ and $\overline{\lambda_n}(\Lambda^*)$ is the $n$-th successive maxima of the lattice $\Lambda^*$ with respect to the norm $M$.

(2) Suppose $\Lambda$ is any $p$-adic lattice of rank $n$ in $k^m$, $\Lambda^*$ is its dual. Let $N$ be a norm on $k^m$. Then we have
$$\overline{\lambda_1}(\Lambda)\cdot\overline{\lambda_n}(\Lambda^*)\geq c',$$
where $\overline{\lambda_1}(\Lambda)=\lambda_1(\Lambda)$ is the first successive maxima of the lattice $\Lambda$ and $\overline{\lambda_n}(\Lambda^*)$ is the $n$-th successive maxima of the lattice $\Lambda^*$ with respect to the norm $N$ and $c'$ is a positive constant depending on $N$.
\end{theorem}

\begin{proof}
Suppose $\Lambda=\mathcal{L}(B)$ and $D=(d_1,\ldots,d_n)$ is the dual basis of $B=(b_1,\ldots,b_n)$. By Proposition \ref{latticedualbasis}, we have $\Lambda^*=\mathcal{L}(D)$. We can choose $d_1,\ldots,d_n$ is an $M$-orthogonal basis of $\Lambda^*$ by Theorem \ref{latticeorthobasis}. Without loss of generality, we may assume $M(d_1)\geq\cdots\geq M(d_n)$, i.e., $M(d_n)=\overline{\lambda_n}(\Lambda^*)$. By the definition of dual basis, we have $b_n\cdot d_n=1$. Thus, we have
$$1=|b_n\cdot d_n|_p\leq M(b_n)\cdot M(d_n)\leq \overline{\lambda_1}(\Lambda)\cdot\overline{\lambda_n}(\Lambda^*).$$
This proves (1). Below, we prove (2).

Suppose $\Lambda=\mathcal{L}(B)$ and $D=(d_1,\ldots,d_n)$ is the dual basis of $B=(b_1,\ldots,b_n)$. By Proposition \ref{latticedualbasis}, we have $\Lambda^*=\mathcal{L}(D)$. We can choose $d_1,\ldots,d_n$ is an $N$-orthogonal basis of $\Lambda^*$ by Theorem \ref{latticeorthobasis}. Without loss of generality, we may assume $N(d_1)\geq\cdots\geq N(d_n)$, i.e., $N(d_n)=\overline{\lambda_n}(\Lambda^*)$. By the definition of dual basis, we have $b_n\cdot d_n=1$. By Lemma \ref{normequiv}, two norms $M$ and $N$ are equivalent, i.e., there are two positive constants $c_1,c_2$ such that
$$c_1N(v)\leq M(v)\leq c_2N(v),\forall v\in k^m.$$ Thus, we have
$$1=|b_n\cdot d_n|_p\leq M(b_n)\cdot M(d_n)\leq c_2^2\cdot N(b_n)\cdot N(d_n)\leq c_2^2\cdot\overline{\lambda_1}(\Lambda)\cdot\overline{\lambda_n}(\Lambda^*).$$
One can put $c'=1/c_2^2$.
\end{proof}

\textbf{Remark.} Theorems \ref{1padictransfer} and \ref{padictransfer} can be viewed as $p$-adic analogues of transference theorems for Euclidean lattices, see \cite{reg}.

\textbf{Example.} This example is from \cite{zdw} section 3. Let $K=\mathbb{Q}_2(\zeta)$ where $\zeta$ is a primitive 5th root of unity. $K$ is an unramified extension of $\mathbb{Q}_2$ of degree 4. Let $N$ be the 2-adic absolute value on $K$. Every element $a$ of $K$ can be written as a unique form $a=a_0+a_1\zeta+a_2\zeta^2+a_3\zeta^3$ with $a_i\in\mathbb{Q}_2(0\leq i\leq3)$. The correspondence $a\longmapsto(a_0,a_1,a_2,a_3)^t$ is a $\mathbb{Q}_2$-linear isomorphism from $K$ onto $\mathbb{Q}_2^4$. Let $\Lambda=\mathcal{L}(1,2\zeta,16\zeta^2+16\zeta^3)$ be a lattice in $K$ of rank 3. We know from \cite{zdw} that $1,2\zeta,16\zeta^2+16\zeta^3$ is an $N$-orthogonal basis of $\Lambda$. With respect to $N$, we have $\overline{\lambda_1}(\Lambda)=1,\overline{\lambda_2}(\Lambda)=1/2,\overline{\lambda_3}(\Lambda)=1/16$. In $\mathbb{Q}_2^4$, the basis matrix $B$ corresponding to $1,2\zeta,16\zeta^2+16\zeta^3$ is
$$B=\left(\begin{array}{lcr}
1&0&0\\
0&2&0\\
0&0&16\\
0&0&16
\end{array}\right).$$
The dual of $B$ is $D=B(B^tB)^{-1}$, one may compute
$$D=\left(\begin{array}{lcr}
1&0&0\\
0&1/2&0\\
0&0&1/32\\
0&0&1/32
\end{array}\right).$$
That is, the dual lattice of $\Lambda$ is $\Lambda^*=\mathcal{L}(1,\frac{1}{2}\zeta,\frac{1}{32}\zeta^2+\frac{1}{32}\zeta^3)$. It is easy to see that $1,\frac{1}{2}\zeta,\frac{1}{32}\zeta^2+\frac{1}{32}\zeta^3$ is an $N$-orthogonal basis of $\Lambda^*$. Thus, with respect to $N$, we have $\overline{\lambda_1}(\Lambda^*)=32,\overline{\lambda_2}(\Lambda^*)=2,\overline{\lambda_3}(\Lambda^*)=1$. We have $\overline{\lambda_1}(\Lambda)\overline{\lambda_3}(\Lambda^*)=1$ and $\overline{\lambda_1}(\Lambda^*)\overline{\lambda_3}(\Lambda)=2$.

\begin{theorem}\label{padicminkowsecond}
(1) Suppose $\Lambda$ is any $p$-adic lattice of rank $n$ in $k^m$. Then we have
$$\prod_{i=1}^n\overline{\lambda_i}(\Lambda)\geq\mbox{det}(\Lambda),$$
where $\overline{\lambda_i}(\Lambda)(1\leq i\leq n)$ are the successive maxima of the lattice $\Lambda$ with respect to the norm $M$.

(2) Suppose $\Lambda$ is any $p$-adic lattice of rank $n$ in $k^m$. Let $N$ be a norm on $k^m$. Then we have
$$\prod_{i=1}^n\overline{\lambda_i}(\Lambda)\geq c''\cdot\mbox{det}(\Lambda),$$
where $\overline{\lambda_i}(\Lambda)(1\leq i\leq n)$ are the successive maxima of the lattice $\Lambda$ with respect to the norm $N$, and $c''$ is a positive constant depending on $N$.
\end{theorem}

\begin{proof}
Suppose $\Lambda=\mathcal{L}(B)$ with $B=(b_1,\ldots,b_n)\in k^{m\times n}$. We can choose $b_1,\ldots,b_n$ is an $M$-orthogonal basis of $\Lambda$ by Theorem \ref{latticeorthobasis}. Then $B^tB$ is an $n\times n$-matrix, whose $(i,j)$-element is $b_i\cdot b_j$, i.e.,
$$B^tB=\left(\begin{array}{lcr}
b_1\cdot b_1&\cdots&b_1\cdot b_n\\
\cdots\cdots&\cdots&\cdots\cdots\\
b_n\cdot b_1&\cdots&b_n\cdot b_n
\end{array}\right).$$
Since $|b_i\cdot b_j|_p\leq M(b_i)\cdot M(b_j)$ and the development of the determinant det$(B^tB)$ is the sum of $n!$ many terms, each term is a product of $n$ many $(b_i\cdot b_j)$'s from different row and different column, we have
$$|\mbox{det}(B^tB)|_p\leq\left(\prod_{i=1}^n\overline{\lambda_i}(\Lambda)\right)^2.$$
Since $\mbox{det}(\Lambda)=|\mbox{det}(B^tB)|_p^{1/2}$, the result follows. This proves (1). Below, we prove (2).

Since $|b_i\cdot b_j|_p\leq M(b_i)\cdot M(b_j)\leq c_2^2N(b_i)\cdot N(b_j)$, here the positive constant $c_2$ is the same as in the proof of Theorem \ref{padicminkowski}, similar to the above argument, we have
$$|\mbox{det}(B^tB)|_p\leq c_2^{2n}\cdot\left(\prod_{i=1}^n\overline{\lambda_i}(\Lambda)\right)^2.$$
Since $\mbox{det}(\Lambda)=|\mbox{det}(B^tB)|_p^{1/2}$, put $c''=1/c_2^n$, the result follows.
\end{proof}

\textbf{Remark.} Theorem \ref{padicminkowsecond} can be viewed as a $p$-adic analogue of Minkowski's second theorem for Euclidean lattices, see \cite{reg,sie,cas}.

\begin{corollary}
(1) Suppose $\Lambda$ is any $p$-adic lattice of rank $n$ in $k^m$, $\Lambda^*$ is its dual. Then we have
$$\prod_{i=1}^n\overline{\lambda_i}(\Lambda)\overline{\lambda_i}(\Lambda^*)\geq1,$$
where $\overline{\lambda_i}(\Lambda),\overline{\lambda_i}(\Lambda^*)(1\leq i\leq n)$ are the successive maxima of the lattice $\Lambda,\Lambda^*$ with respect to the norm $M$, respectively.

(2) Suppose $\Lambda$ is any $p$-adic lattice of rank $n$ in $k^m$, $\Lambda^*$ is its dual. Let $N$ be a norm on $k^m$. Then we have
$$\prod_{i=1}^n\overline{\lambda_i}(\Lambda)\overline{\lambda_i}(\Lambda^*)\geq (c'')^2,$$
where $\overline{\lambda_i}(\Lambda),\overline{\lambda_i}(\Lambda^*)(1\leq i\leq n)$ are the successive maxima of the lattice $\Lambda,\Lambda^*$ with respect to the norm $N$, respectively, and the positive constant $c''$ is the same as in Theorem \ref{padicminkowsecond} (2).
\end{corollary}

\begin{proof}
This follows from Theorem \ref{padicminkowsecond} and Proposition \ref{dualdet}.
\end{proof}

\section{Computation of $p$-adic Minkowski's constants}\label{computconst}
A natural question is that whether the three positive constants $c,c',c''$ in Theorems \ref{padicminkowski}, \ref{padictransfer} and \ref{padicminkowsecond} are computable. Checking the proofs of these three theorems, we observe that the problem is reduced to the computation of the constant $c_2$ in Lemma \ref{normequiv}. In this section, we show that, by invoking the CVP algorithm of $p$-adic lattices in \cite{deng2}, the constant $c_2$ in Lemma \ref{normequiv} is computable, hence the three positive constants $c,c',c''$ in Theorems \ref{padicminkowski}, \ref{padictransfer} and \ref{padicminkowsecond} are computable. The main result is the following algorithmic version of Lemma \ref{normequiv}.

\begin{theorem}
Let $k=\mathbb{Q}_p$ or $k=\mathbb{F}_p((T))$. Let $V$ be a vector space over $k$ of finite dimension $m>0$. Let $N$ be a norm on $V$. Assume we can efficiently compute the norm $N(v)$ for every $v\in V$. Fix a basis $e_1,\ldots,e_m$ of $V$ over $k$. For $v=\sum_{i=1}^mv_ie_i\in V$ with $v_i\in k(1\leq i\leq m)$, define $M(v)=\max_{1\leq i\leq m}|v_i|_p$. $M$ is a norm on $V$. Then the two norms $M$ and $N$ are equivalent, i.e., there are two computable positive constants $c_1,c_2$ such that
$$c_1N(v)\leq M(v)\leq c_2N(v),\forall v\in V.$$
\end{theorem}

\begin{proof}
If $m=1$, since $N(v)=N(v_1e_1)=N(e_1)\cdot|v_1|_p=N(e_1)M(v)$, the result is clear. Suppose $m>1$.

For $v=\sum_{i=1}^mv_ie_i\in V$ with $v_i\in k(1\leq i\leq m)$, we have
$$N(v)=N\left(\sum_{i=1}^mv_ie_i\right)\leq\max_{1\leq i\leq m}N(v_ie_i)\leq\max_{1\leq i\leq m}N(e_i)\cdot\max_{1\leq i\leq m}|v_i|_p.$$
That is $N(v)\leq C\cdot M(v)$ with $C=\max_{1\leq i\leq m}N(e_i)$. Obviously, the constant $C$ is computable.

On the other hand, for $v\neq0$, then not all $v_i(1\leq i\leq m)$ are zero. Suppose $|v_j|_p=\max_{1\leq i\leq m}|v_i|_p$ for some $j\in\{1,\ldots,m\}$. Then
\begin{equation}\label{normmn}
N(v)=N\left(v_j\cdot\frac{v}{v_j}\right)=|v_j|_p\cdot N\left(\frac{v}{v_j}\right)=M(v)\cdot N\left(\frac{v}{v_j}\right).
\end{equation}
Since $\frac{v_i}{v_j}\in\mathcal{O}_k$ for each $i=1,\ldots,m$, we see that
$$\frac{v}{v_j}\in e_j+\mathcal{L}(e_1,\ldots,e_{j-1},e_{j+1},\ldots,e_m).$$
Letting $j=1,\ldots,m$, by solving $m$ many CVP-instances with the rank-$(m-1)$ lattice $\mathcal{L}(e_1,\ldots,e_{j-1},e_{j+1},\ldots,e_m)$, the target vector $e_j$ and the norm $N$ via the algorithm in \cite{deng2}, we can compute the distances dist$(e_j,\mathcal{L}(e_1,\ldots,e_{j-1},e_{j+1},\ldots,e_m))$ for $1\leq j\leq m$. Put
$$c=\min_{1\leq j\leq m}\mbox{dist}(e_j,\mathcal{L}(e_1,\ldots,e_{j-1},e_{j+1},\ldots,e_m)).$$
The constant $c$ is clearly computable. By equation (\ref{normmn}), we have
$$N(v)\geq c\cdot M(v),\forall v\in V.$$
The proof is complete.

\end{proof}

Usually, as in \cite{deng1,deng2}, we take $V$ be an extension field of $k$ of finite degree, and $N$ be the $p$-adic absolute value on $V$. So, making the constants $c,c',c''$ in Theorems \ref{padicminkowski}, \ref{padictransfer} and \ref{padicminkowsecond} computable for a general norm $N$ is useful.

\vspace{0.5cm}
\textbf{Acknowledgments}: We thank Dr. Zhaonan Wang and Dr. Chi Zhang for helpful discussion. This work was supported by National Natural Science Foundation of China (No. 12271517) and National Key Research and Development Project of China (No. 2018YFA0704705).

\end{document}